\documentclass[]{article}
\usepackage{amsmath,amsthm,amssymb}
\title{The dimension of Kronheimer-Mrowka instanton homology group for plane trivalent graphs}
\author{Zipei Zhuang}

\usepackage{pifont}
\usepackage{indentfirst}
\usepackage{setspace}
\usepackage{mathrsfs}
\usepackage[backref]{hyperref}

\usepackage{subfigure}
\usepackage{graphicx}
\usepackage{import}
\usepackage{xifthen}
\usepackage{pdfpages}
\usepackage{transparent}
\newtheorem{theorem}{Theorem}%
\newtheorem{lemma}{Lemma}%
\newtheorem{corollary}{Corollary}%
\newtheorem{proposition}{Proposition}

\begin{document}
	
	\maketitle
	
	\begin{abstract}
		We proved that the dimension of the  $\mathbb{F}$-vector space $J^\#(G)$ for a plane trivalent graph $G$, defined by Kronheimer and Mrowka in \cite{MR3880205}, is equal to the number of Tait colorings of $G$.
	\end{abstract}

\section{Introduction}
Let $G$ be a trivalent graph. A $\mathbf{Tait \ coloring}$ of $G$ is a function from the edges of $G$ to a 3-element set of "colors" \{1,2,3\} such that edges of three different colors are incident at each vertex. The famous four-color theorem is euqivalent to that every bridgeless planar trivalent graph admits a Tait coloring.

In \cite{MR3880205}, Kronheimer and Mrowka used their $SO(3)$ instanton Floer homology to construct an $\mathbb{F}$-vector space $J^\#(G)$ ($\mathbb{F}$ is the field of 2 elements) for each trivalent graph in $\mathbb{R}^3$. They conjectured that when $G$ is planar, the dimension of $J^\#(G)$ is equal to the number $T(G)$ of Tait colorings of $G$. If it were true, the non-vanishing theorem they proved (\cite{MR3880205} Theorem 1.1) which says that $J^\#(G)$ is zero if and only if $G$ has an embedded bridge, would provide an alternative proof of the four-color theorem.

$J^\#(G)$ is a indeed a functor, from the category of webs and foams to the category of $\mathbb{F}$-vector space, satisfying a series of local properties analogous to Khovanov's $sl_3$ homology  $H_3$ for MOY graphs\cite{MR2100691}.
Indeed, they also proposed a combinatorial counterpart $J^b$ to $J^\#$ for plane trivalent graphs by imitating the construction of $sl_3$ homology, whose well-definedness is proved in \cite{MR4178907}.

 Khovanov used the invariant $H_3$ to construct a homology theory for knots. More specifically, given a knot diagram $D$, he resolutes  each crossing point in two ways and get $2^n$ MOY graphs, where $n$ is the number of crossings in $D$. The $sl_3$ homology groups of these $2^n$ graphs are organized into a complex $C_3(D)$. Using the local relations of $H_3$ on  foams he proved that the chain homotopy type of $C_3(D)$ is independent of the choice of $D$, i.e., invariant under the Reidemeister moves. In particular, the homology and the Euler characteristic of the complex are invariants of the knot.   The construction of $C_3(D)$ can be explained by the picture:

	\begin{figure}[h]
	\def\svgwidth{\columnwidth}
	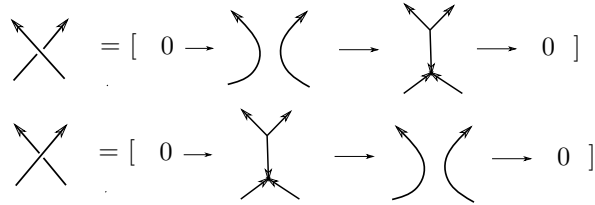
	\caption{Local relation of the complex of Khovanov's $sl_3$ knot homology} 
	\label{graph1}
\end{figure}

This construction motivates us to prove the conjecture via Penrose's approach to  the number of Tait colorings(\cite{MR0281657}, see also \cite{jaeger1989penrose}). The Penrose number, which is defined for plane trivalent graphs “with crossings”, can be defined through a “skein relation”(see \ref{prop2}), and is equal to the number of Tait colorings when the graph has no crossings.

From now on we use $J$ to denote the $J^\#$ in \cite{MR3880205}, suggesting that our methods also works for $J^b$ in that paper.

\begin{theorem}\label{thm1}
	Let $G$ be a plane trivalent graph ,  then the dimension of $J(G)$ is equal to the number of Tait colorings of $G$.
	
\end{theorem}

 Our strategy to prove Theorem  \ref{thm1}     is to generalize the definition of $J$ to any plane trivalent graph with crossings,  just as Khovanov's definition of knot homology group from the $sl_3$ homology of MOY graphs, and to prove that the Euler characteristic of $J$ satisfies the relation of Penrose number in  Prop. \ref{prop2}.

 In Section \ref{sec2} we review the concept of plane trivalent graphs with crossings, and the basic properties of Penrose number. In Section \ref{sec3} we review the local relations of  $J$ proved in \cite{MR3880205}, which is needed in Section \ref{sec4}  when proving the chain homotopy invariance.   Then in Section \ref{sec4} we generalize $J$ to any plane trivalent graphs with crossings, for which the well-definedness follows  from a similar argument for proving that Khovanov's $sl_3$ knot homology is well-defined.

\section{The Penrose number}  \label{sec2}

 In this paper a \textbf{web} is a plane trivalent graph. We need a more general concept.
 
 Let $G$ be a trivalent graph. We can draw $G$ on the plane, with a finite number of crossing points. We call such a drawing  a \textbf{virtual web}, with underlying graph $G$. Two virtual webs are \textbf{equivalent} if they are related by a sequence of plane isotopies and \textbf{virtual Reidemeister moves} as in Fig   \ref{graph1}   .
 
 An immediate observation is:
 \begin{lemma}  \label{l1}
 	If a virtual web has no vertices, then it is (equivalent to) a collection of circles. 
 \end{lemma}
 
 \begin{figure}[h]
 	\def\svgwidth{\columnwidth}
\begingroup%
  \makeatletter%
  \providecommand\color[2][]{%
    \errmessage{(Inkscape) Color is used for the text in Inkscape, but the package 'color.sty' is not loaded}%
    \renewcommand\color[2][]{}%
  }%
  \providecommand\transparent[1]{%
    \errmessage{(Inkscape) Transparency is used (non-zero) for the text in Inkscape, but the package 'transparent.sty' is not loaded}%
    \renewcommand\transparent[1]{}%
  }%
  \providecommand\rotatebox[2]{#2}%
  \newcommand*\fsize{\dimexpr\f@size pt\relax}%
  \newcommand*\lineheight[1]{\fontsize{\fsize}{#1\fsize}\selectfont}%
  \ifx\svgwidth\undefined%
    \setlength{\unitlength}{1190.5511811bp}%
    \ifx\svgscale\undefined%
      \relax%
    \else%
      \setlength{\unitlength}{\unitlength * \real{\svgscale}}%
    \fi%
  \else%
    \setlength{\unitlength}{\svgwidth}%
  \fi%
  \global\let\svgwidth\undefined%
  \global\let\svgscale\undefined%
  \makeatother%
  \begin{picture}(1,0.38095238)%
    \lineheight{1}%
    \setlength\tabcolsep{0pt}%
    \put(0,0){\includegraphics[width=\unitlength,page=1]{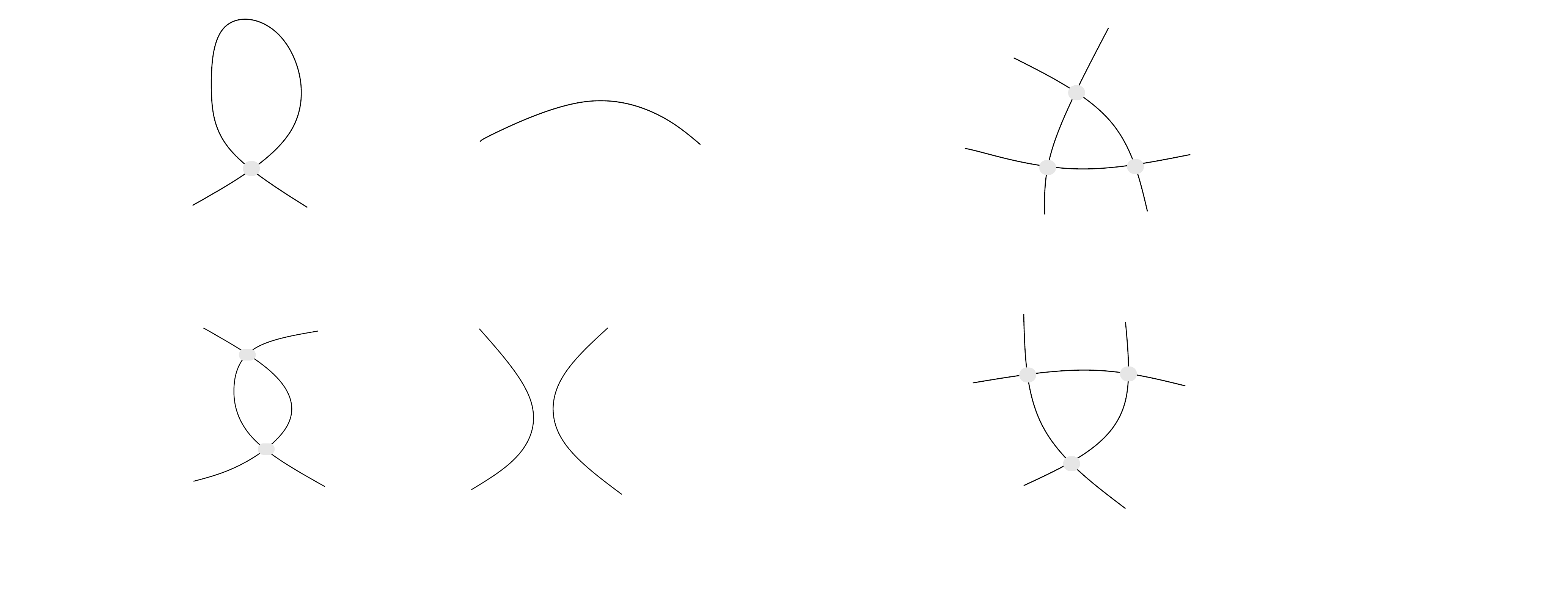}}%
    \put(0.22654274,0.21613617){\color[rgb]{0,0,0}\makebox(0,0)[lt]{\lineheight{1.25}\smash{\begin{tabular}[t]{l}virtual RV I\end{tabular}}}}%
    \put(0.19844017,0.01266588){\color[rgb]{0,0,0}\makebox(0,0)[lt]{\lineheight{1.25}\smash{\begin{tabular}[t]{l}virtual RV II\end{tabular}}}}%
    \put(0.60952613,0.01775676){\color[rgb]{0,0,0}\makebox(0,0)[lt]{\lineheight{1.25}\smash{\begin{tabular}[t]{l}virtual RV III\end{tabular}}}}%
    \put(0,0){\includegraphics[width=\unitlength,page=2]{2.1.pdf}}%
  \end{picture}%
\endgroup%

 	\caption{Virtual Reidemeister moves} 
 	\label{graph1}
 \end{figure}

 From now on when we say a virtual web, we mean an equivalence class of virtual webs. See Fig \ref{graph2}     for 3 different virtual webs for the $\theta$-graph.  Note that the cyclic order of three edges at each vertex is an invariant of a virtual web.

 	\begin{figure}[h]
 	\def\svgwidth{\columnwidth}
 	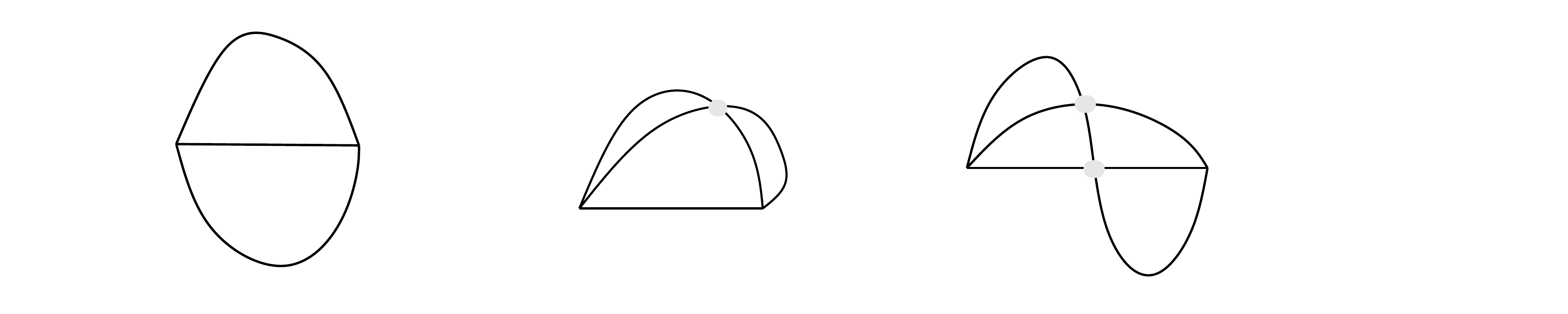
 	\caption{Three virtual webs with underlying graph $\theta$} 
 	\label{graph2}
 \end{figure}

Let $G$ be a trivalent graph, and $D$ is a virtual web for $G$. Denote by $T(G)$ the set of Tait colorings of $G$. Suppose $f \in T(G)$.  
 A vertex $v$ of $D$ is said to be \textbf{positive} with respect to $f$ if the colors incident to $v$ are 1,2,3 in the clockwise order, otherwise it is \textbf{negative}. Let $n^+(f), n^-(f)$ be the number of positive and negative vertices of $D$ with respect to $f$. Define
\begin{equation}
	s_D(f)=\begin{cases}
		1 \text{  \quad if } n^+(f) \equiv n^-(f) \text{ mod 4} \\
		-1 \text{ \quad otherwise}
	\end{cases}
\end{equation}

The \textbf{Penrose number} $P(D)$ of $D$ is defined to be the sum of $s_D(f)$ over all Tait colorings $f$ of $G$. It is obviously invariant under the virtual Reidemeister moves.

\begin{proposition}(\cite{jaeger1989penrose}, Propersition 1)  \label{prop3}
     Let $D$ be a web with underlying graph $G$, then $T(G)=P(D)$.  
\end{proposition}

\begin{proposition}(\cite{jaeger1989penrose}, Propersition 2) \label{prop2}
     Let $D, D', D''$ be three virtual webs which differ only within a small circle as in Fig  \ref{graph3}  . Then $P(D)=P(D')-P(D'')$.
\end{proposition}
Obviously this relation together with the fact that $P(U_n)=3^n$ where $U_n$ is the n-component unlink, uniquely characterize $P$.

	\begin{figure}[h]
	\def\svgwidth{\columnwidth}
\begingroup%
  \makeatletter%
  \providecommand\color[2][]{%
    \errmessage{(Inkscape) Color is used for the text in Inkscape, but the package 'color.sty' is not loaded}%
    \renewcommand\color[2][]{}%
  }%
  \providecommand\transparent[1]{%
    \errmessage{(Inkscape) Transparency is used (non-zero) for the text in Inkscape, but the package 'transparent.sty' is not loaded}%
    \renewcommand\transparent[1]{}%
  }%
  \providecommand\rotatebox[2]{#2}%
  \newcommand*\fsize{\dimexpr\f@size pt\relax}%
  \newcommand*\lineheight[1]{\fontsize{\fsize}{#1\fsize}\selectfont}%
  \ifx\svgwidth\undefined%
    \setlength{\unitlength}{850.39370079bp}%
    \ifx\svgscale\undefined%
      \relax%
    \else%
      \setlength{\unitlength}{\unitlength * \real{\svgscale}}%
    \fi%
  \else%
    \setlength{\unitlength}{\svgwidth}%
  \fi%
  \global\let\svgwidth\undefined%
  \global\let\svgscale\undefined%
  \makeatother%
  \begin{picture}(1,0.26666667)%
    \lineheight{1}%
    \setlength\tabcolsep{0pt}%
    \put(0,0){\includegraphics[width=\unitlength,page=1]{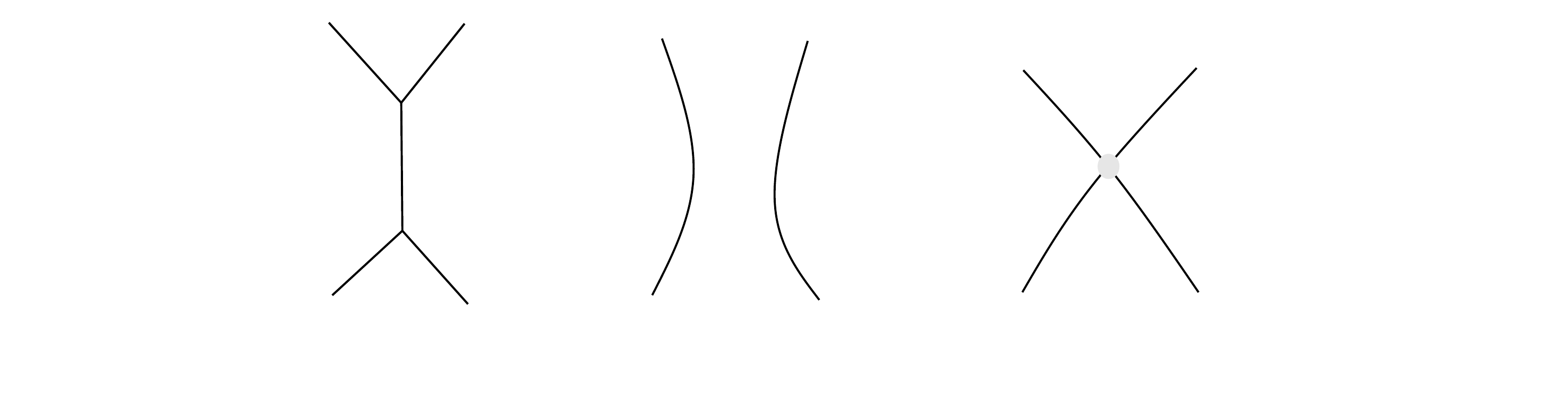}}%
    \put(0.2367407,0.01893135){\color[rgb]{0,0,0}\makebox(0,0)[lt]{\lineheight{1.25}\smash{\begin{tabular}[t]{l}D\end{tabular}}}}%
    \put(0.45768346,0.0184859){\color[rgb]{0,0,0}\makebox(0,0)[lt]{\lineheight{1.25}\smash{\begin{tabular}[t]{l}D'\end{tabular}}}}%
    \put(0.69778066,0.01893135){\color[rgb]{0,0,0}\makebox(0,0)[lt]{\lineheight{1.25}\smash{\begin{tabular}[t]{l}D"\end{tabular}}}}%
  \end{picture}%
\endgroup%

	\caption{Local relation of Penrose number} 
	\label{graph3}
\end{figure}
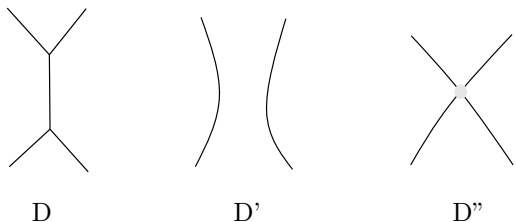

\section{Local relations of the functor $J$}  \label{sec3}

   A \textbf{closed pre-foam} $F$ is a compact 2-dimensional CW-complex such that each point has an open neighborhood that is either an open disk, the product of a tripod and an open interval, or the cone over the 1-skeleton of a tetrahedron. The subspace $s(F)$ of points of tha last two kinds is a four-valent graph. Closed prefoams can be decorated by finite number of points(called dots) on their facets. These dots can float on facets but cannot cross $s(F)$.
   
   A \textbf{closed  foam} $F$ is a closed pre-foam together with a piece-wise linear embedding into $\mathbb{R}^4$. A 3-dimensional space $T \cong \mathbb{R}^3 \subset \mathbb{R}^4 $  intersects a closed foam $F$ generically if $F \cap T$ is a trivalent graph $G$ in $T$, and for a tubular neighborhoood $N$ of $T$, $(N \cap F, N)$ is piecewise-linearly homeomorphic to $(G \times (-\epsilon, \epsilon), \mathbb{R}^3 \times (-\epsilon, \epsilon))$.
   
   A \textbf{foam}(with boundary) is the intersection of a closed foam $F$ with $\mathbb{R}^3 \times [0,1]$ in $\mathbb{R}^4$ such that $F$ generically intersects $\mathbb{R}^3 \times \{0\}$ and $\mathbb{R}^3 \times \{1\}$. We can view a foam $F$ as a cobordism from $\partial_0F=F \cap \mathbb{R}^3 \times \{0\}$ to $\partial_1F=F \cap \mathbb{R}^3 \times \{1\}$. In particular, a closed foam is a foam withoout boundary and is a cobordism from the empty graph to itself. 
   
   If $F, G$ are 2 foams such that the webs $\partial_0F, \partial_1G$ are identical, define the composition $GF$ by concatenating along their common boundary. In this way we obtain a category $Foams$, with webs as objects and isomorphism classes of foams with boundary as morphisms.
   
   In \cite{MR3880205}, Kronheimer and Mrowka defined a functor 
   \begin{equation}
   	J^\# :  Foams \longrightarrow Vect_\mathbb{F}
   \end{equation}
where $Vect_\mathbb{F}$ is the category of vector spaces over the 2-elements field $\mathbb{F}$. We list the properties of $J^\#$ we need afterwards.

\begin{proposition}
	For the empty graph $\emptyset$, $J(\emptyset)= \mathbb{F}$.
\end{proposition}
If $F$ is a closed foam, then $J(F): \mathbb{F} \longrightarrow \mathbb{F}$ is characterized by the image of $1 \in \mathbb{F} \cong J(\emptyset)$. Still denote this scalar by $J(F)$.

\begin{proposition}(\cite{MR3880205}, Prop. 5.1)
	For a circle $U$, $J(\emptyset)= \mathbb{F}^3$.
\end{proposition}

\begin{proposition}(\cite{MR3880205}, Prop. 5.2)
	If a foam $F$ has 3 or more dots on one of its facets, then $J(F)=0.$  
\end{proposition}

\begin{proposition}(\cite{MR3880205}, Corollary 4.4)
	If $W=W_1 \sqcup W_2$ is a disjoint unnion of 2 webs, then $J(W)= J(W_1) \otimes J(W_2).$
\end{proposition}

\begin{corollary}\label{c1}
	If $W$ is a disjoint union of $k$ circles, then $J(F)= \mathbb{F}^{3k}$.
\end{corollary}

\begin{proposition}(\cite{MR3880205}, Prop 5.3, Prop 6.2)
	If $F$ is an unknotted closed 2-manifold, then
	\begin{equation}
		J(F)=\left\{ \begin{aligned}
			1 & \quad F \text{ is the sphere with 2 dots, or the torus without dots} \\
			0 & \quad \text{ otherwise}
		\end{aligned}
	\right.
	\end{equation}
\end{proposition}

\begin{proposition} (\cite{MR3880205}, Prop 6.1) The foam evaluation satisfies the neck-cutting relation
	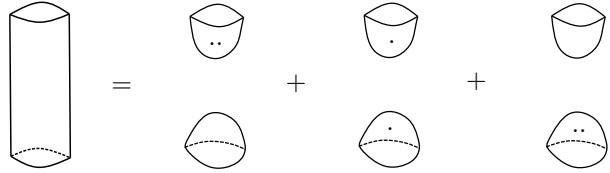
\begin{figure}[h] \label{6.1}
		\def\svgwidth{\columnwidth}
\begingroup%
  \makeatletter%
  \providecommand\color[2][]{%
    \errmessage{(Inkscape) Color is used for the text in Inkscape, but the package 'color.sty' is not loaded}%
    \renewcommand\color[2][]{}%
  }%
  \providecommand\transparent[1]{%
    \errmessage{(Inkscape) Transparency is used (non-zero) for the text in Inkscape, but the package 'transparent.sty' is not loaded}%
    \renewcommand\transparent[1]{}%
  }%
  \providecommand\rotatebox[2]{#2}%
  \newcommand*\fsize{\dimexpr\f@size pt\relax}%
  \newcommand*\lineheight[1]{\fontsize{\fsize}{#1\fsize}\selectfont}%
  \ifx\svgwidth\undefined%
    \setlength{\unitlength}{1133.85826772bp}%
    \ifx\svgscale\undefined%
      \relax%
    \else%
      \setlength{\unitlength}{\unitlength * \real{\svgscale}}%
    \fi%
  \else%
    \setlength{\unitlength}{\svgwidth}%
  \fi%
  \global\let\svgwidth\undefined%
  \global\let\svgscale\undefined%
  \makeatother%
  \begin{picture}(1,0.2)%
    \lineheight{1}%
    \setlength\tabcolsep{0pt}%
    \put(0,0){\includegraphics[width=\unitlength,page=1]{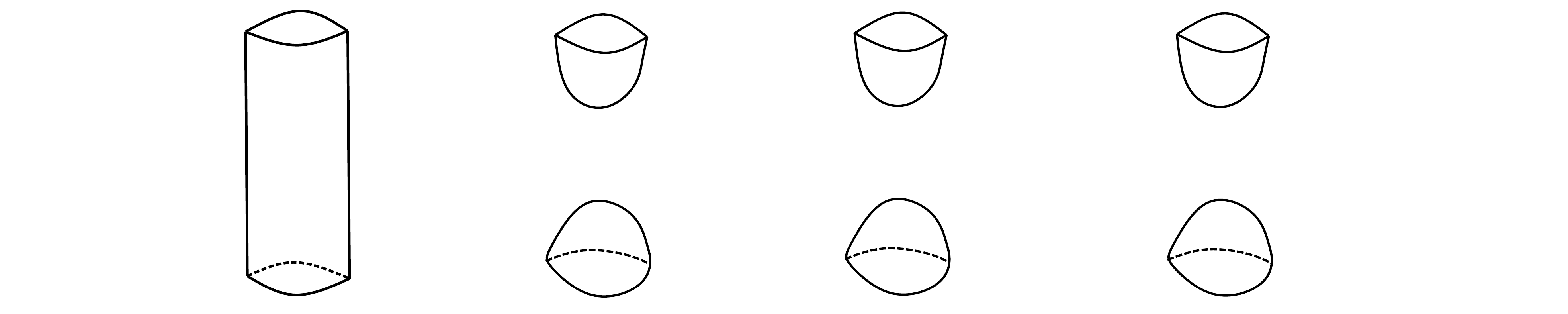}}%
    \put(0.26796171,0.0949722){\color[rgb]{0,0,0}\makebox(0,0)[lt]{\lineheight{1.25}\smash{\begin{tabular}[t]{l}=\end{tabular}}}}%
    \put(0.4594059,0.09675541){\color[rgb]{0,0,0}\makebox(0,0)[lt]{\lineheight{1.25}\smash{\begin{tabular}[t]{l}+\end{tabular}}}}%
    \put(0.65691697,0.09853863){\color[rgb]{0,0,0}\makebox(0,0)[lt]{\lineheight{1.25}\smash{\begin{tabular}[t]{l}+\end{tabular}}}}%
    \put(0.37439571,0.14572039){\color[rgb]{0,0,0}\makebox(0,0)[lt]{\lineheight{1.25}\smash{\begin{tabular}[t]{l}..\end{tabular}}}}%
    \put(0.57136102,0.14874838){\color[rgb]{0,0,0}\makebox(0,0)[lt]{\lineheight{1.25}\smash{\begin{tabular}[t]{l}.\end{tabular}}}}%
    \put(0.56970313,0.05245474){\color[rgb]{0,0,0}\makebox(0,0)[lt]{\lineheight{1.25}\smash{\begin{tabular}[t]{l}.\end{tabular}}}}%
    \put(0.77289531,0.05115406){\color[rgb]{0,0,0}\makebox(0,0)[lt]{\lineheight{1.25}\smash{\begin{tabular}[t]{l}..\end{tabular}}}}%
  \end{picture}%
\endgroup%
 
		\caption{The neck-cutting relation}
		\label{graph6.1}
	\end{figure}
\end{proposition}

\begin{proposition}(Implicited in Proof of Prop 6.5,\cite{MR3880205})
	\begin{figure}[h]\label{6.2}
		\def\svgwidth{\columnwidth}
\begingroup%
  \makeatletter%
  \providecommand\color[2][]{%
    \errmessage{(Inkscape) Color is used for the text in Inkscape, but the package 'color.sty' is not loaded}%
    \renewcommand\color[2][]{}%
  }%
  \providecommand\transparent[1]{%
    \errmessage{(Inkscape) Transparency is used (non-zero) for the text in Inkscape, but the package 'transparent.sty' is not loaded}%
    \renewcommand\transparent[1]{}%
  }%
  \providecommand\rotatebox[2]{#2}%
  \newcommand*\fsize{\dimexpr\f@size pt\relax}%
  \newcommand*\lineheight[1]{\fontsize{\fsize}{#1\fsize}\selectfont}%
  \ifx\svgwidth\undefined%
    \setlength{\unitlength}{907.08661417bp}%
    \ifx\svgscale\undefined%
      \relax%
    \else%
      \setlength{\unitlength}{\unitlength * \real{\svgscale}}%
    \fi%
  \else%
    \setlength{\unitlength}{\svgwidth}%
  \fi%
  \global\let\svgwidth\undefined%
  \global\let\svgscale\undefined%
  \makeatother%
  \begin{picture}(1,0.21875)%
    \lineheight{1}%
    \setlength\tabcolsep{0pt}%
    \put(0,0){\includegraphics[width=\unitlength,page=1]{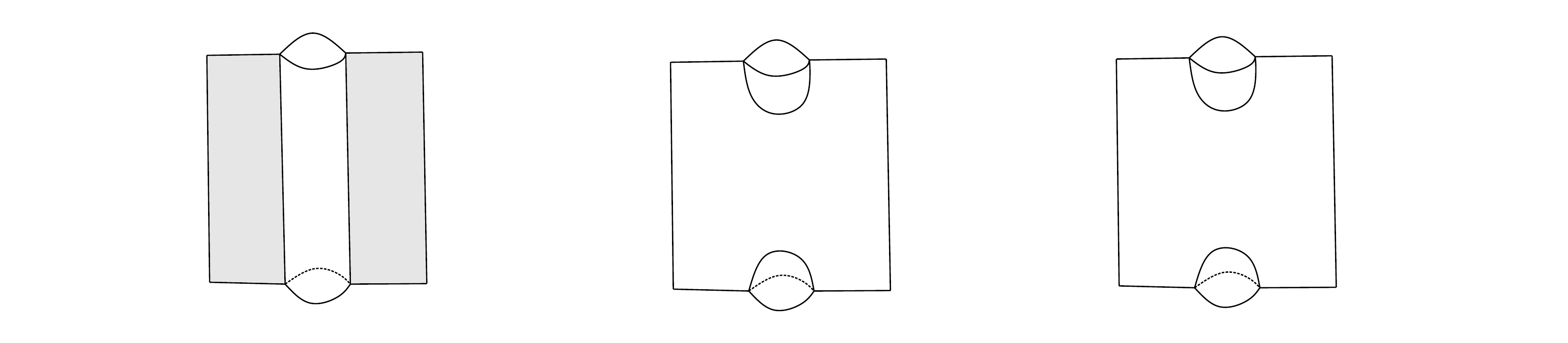}}%
    \put(0.49212508,0.17960226){\color[rgb]{0,0,0}\makebox(0,0)[lt]{\lineheight{1.25}\smash{\begin{tabular}[t]{l}.\end{tabular}}}}%
    \put(0.49588361,0.03051599){\color[rgb]{0,0,0}\makebox(0,0)[lt]{\lineheight{1.25}\smash{\begin{tabular}[t]{l}.\end{tabular}}}}%
    \put(0.77840328,0.15940661){\color[rgb]{0,0,0}\makebox(0,0)[lt]{\lineheight{1.25}\smash{\begin{tabular}[t]{l}.\end{tabular}}}}%
    \put(0.77923851,0.05020193){\color[rgb]{0,0,0}\makebox(0,0)[lt]{\lineheight{1.25}\smash{\begin{tabular}[t]{l}.\end{tabular}}}}%
    \put(0,0){\includegraphics[width=\unitlength,page=2]{6.2.pdf}}%
    \put(0.32625332,0.09987098){\color[rgb]{0,0,0}\makebox(0,0)[lt]{\lineheight{1.25}\smash{\begin{tabular}[t]{l}=\end{tabular}}}}%
    \put(0.62213783,0.09750862){\color[rgb]{0,0,0}\makebox(0,0)[lt]{\lineheight{1.25}\smash{\begin{tabular}[t]{l}+\end{tabular}}}}%
  \end{picture}%
\endgroup%
 
		\caption{The bigon relation}
		\label{graph6.2}
	\end{figure}
\end{proposition}

\begin{proposition}(\cite{MR3880205}, Prop 5.6)
	For the theta foam with dots $\theta(k_1, k_2, k_3)$,
	\begin{equation}
		J(\theta(k_1, k_2, k_3))= \left\{ \begin{aligned}
			1 & \quad \text{if } \{k_1, k_2, k_3\}= \{0, 1, 2\} \\
			0 & \quad \text{otherwise}
		\end{aligned}
	\right.
	\end{equation}
\end{proposition}

\begin{proposition}(Implicited in Proof of Prop 6.8,\cite{MR3880205})
	\begin{figure}[h]\label{6.3}
		\def\svgwidth{\columnwidth}
\begingroup%
  \makeatletter%
  \providecommand\color[2][]{%
    \errmessage{(Inkscape) Color is used for the text in Inkscape, but the package 'color.sty' is not loaded}%
    \renewcommand\color[2][]{}%
  }%
  \providecommand\transparent[1]{%
    \errmessage{(Inkscape) Transparency is used (non-zero) for the text in Inkscape, but the package 'transparent.sty' is not loaded}%
    \renewcommand\transparent[1]{}%
  }%
  \providecommand\rotatebox[2]{#2}%
  \newcommand*\fsize{\dimexpr\f@size pt\relax}%
  \newcommand*\lineheight[1]{\fontsize{\fsize}{#1\fsize}\selectfont}%
  \ifx\svgwidth\undefined%
    \setlength{\unitlength}{992.12598425bp}%
    \ifx\svgscale\undefined%
      \relax%
    \else%
      \setlength{\unitlength}{\unitlength * \real{\svgscale}}%
    \fi%
  \else%
    \setlength{\unitlength}{\svgwidth}%
  \fi%
  \global\let\svgwidth\undefined%
  \global\let\svgscale\undefined%
  \makeatother%
  \begin{picture}(1,0.22857143)%
    \lineheight{1}%
    \setlength\tabcolsep{0pt}%
    \put(0,0){\includegraphics[width=\unitlength,page=1]{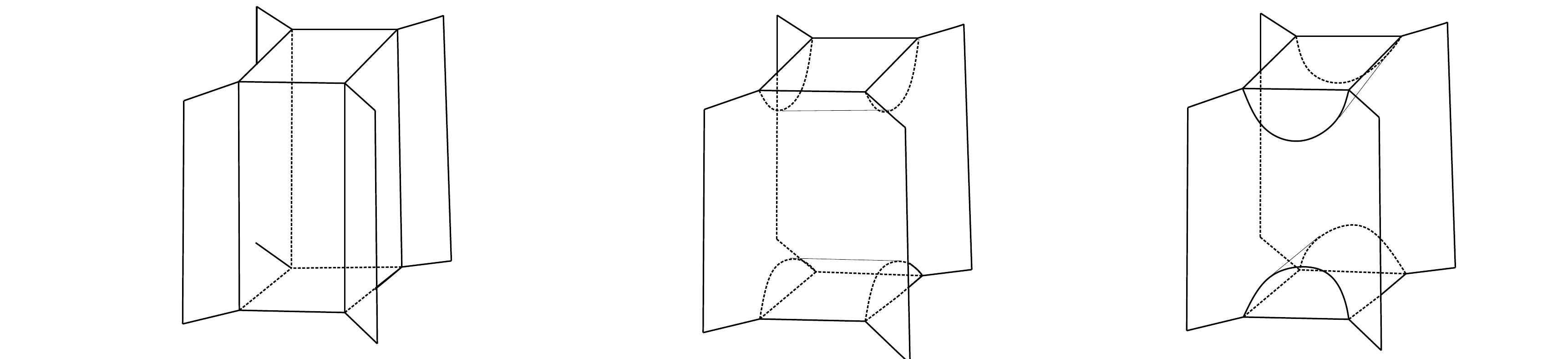}}%
    \put(0.34781521,0.10039597){\color[rgb]{0,0,0}\makebox(0,0)[lt]{\lineheight{1.25}\smash{\begin{tabular}[t]{l}=\end{tabular}}}}%
    \put(0.66828052,0.10176433){\color[rgb]{0,0,0}\makebox(0,0)[lt]{\lineheight{1.25}\smash{\begin{tabular}[t]{l}+\end{tabular}}}}%
  \end{picture}%
\endgroup%
 
		\caption{The square relation}
		\label{graph6.3}
	\end{figure}
\end{proposition}

Figure \ref{6.4}  \ref{6.4.2} illustrates some relations which are corollaries of  \ref{6.1} \ref{6.2}  \ref{6.3}

\section{Generalization to virtual webs}  \label{sec4}

From now on we consider oriented virtual webs, i.e., virtual webs with an orientation on each of its edges. 
The sole role that the orientations play is to specify the resolutions at the crossings, as stated below. We can resolute each virtual crossing in two ways, see Fig  \ref{3.1}. If we resolute every  crossing, we get a web(ignore the orientations), and we call it a \textbf{complete resolution} of the virtual web.

\begin{figure}[h]
	\def\svgwidth{\columnwidth}
\begingroup%
  \makeatletter%
  \providecommand\color[2][]{%
    \errmessage{(Inkscape) Color is used for the text in Inkscape, but the package 'color.sty' is not loaded}%
    \renewcommand\color[2][]{}%
  }%
  \providecommand\transparent[1]{%
    \errmessage{(Inkscape) Transparency is used (non-zero) for the text in Inkscape, but the package 'transparent.sty' is not loaded}%
    \renewcommand\transparent[1]{}%
  }%
  \providecommand\rotatebox[2]{#2}%
  \newcommand*\fsize{\dimexpr\f@size pt\relax}%
  \newcommand*\lineheight[1]{\fontsize{\fsize}{#1\fsize}\selectfont}%
  \ifx\svgwidth\undefined%
    \setlength{\unitlength}{841.88976378bp}%
    \ifx\svgscale\undefined%
      \relax%
    \else%
      \setlength{\unitlength}{\unitlength * \real{\svgscale}}%
    \fi%
  \else%
    \setlength{\unitlength}{\svgwidth}%
  \fi%
  \global\let\svgwidth\undefined%
  \global\let\svgscale\undefined%
  \makeatother%
  \begin{picture}(1,0.16835017)%
    \lineheight{1}%
    \setlength\tabcolsep{0pt}%
    \put(0,0){\includegraphics[width=\unitlength,page=1]{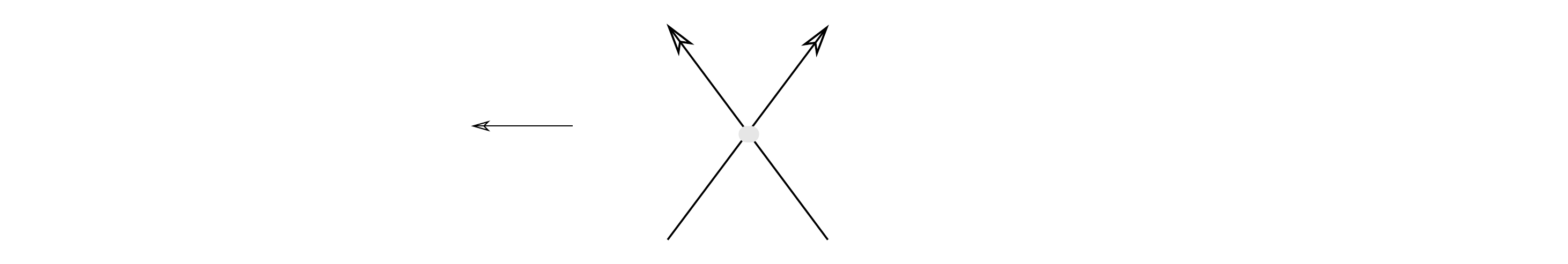}}%
    \put(0.2926392,0.09621179){\color[rgb]{0,0,0}\makebox(0,0)[lt]{\lineheight{1.25}\smash{\begin{tabular}[t]{l}0-resolution\\\end{tabular}}}}%
    \put(0,0){\includegraphics[width=\unitlength,page=2]{3.1.pdf}}%
    \put(0.55896206,0.09359399){\color[rgb]{0,0,0}\makebox(0,0)[lt]{\lineheight{1.25}\smash{\begin{tabular}[t]{l}1-resolution\\\end{tabular}}}}%
    \put(0,0){\includegraphics[width=\unitlength,page=3]{3.1.pdf}}%
  \end{picture}%
\endgroup%

	\caption{Resolute a  crossing of a virtual web} 
	\label{3.1}
\end{figure}

Let $I$ be the set of  crossings of a virtual web $D$. To $D$ we associate an $|I|$-dimensional cube of $\mathbb{F}$-vector spaces. Vertices of the cube are in a bijection with subsets of $I$. To $J\subset I$ we associate the complete resolution $D_J$ of $D$ according to $J$: a  crossing $p \in I$ gets 1-resolution if and only if $p \in J$.

In the vertex $J$ we place the graded $\mathbb{F}$ -vector space $J(D_J)$, and to an inclusion $J \subset J \cup \{b\}$, where $b \in I\backslash J$, we assign the homomorphism
\begin{equation}
	J(D_J) \longrightarrow J(D_{J \cup \{b\}})
\end{equation}
induced by the basic cobordism in Fig \ref{3.2}.

\begin{figure}
	\def\svgwidth{\columnwidth}
	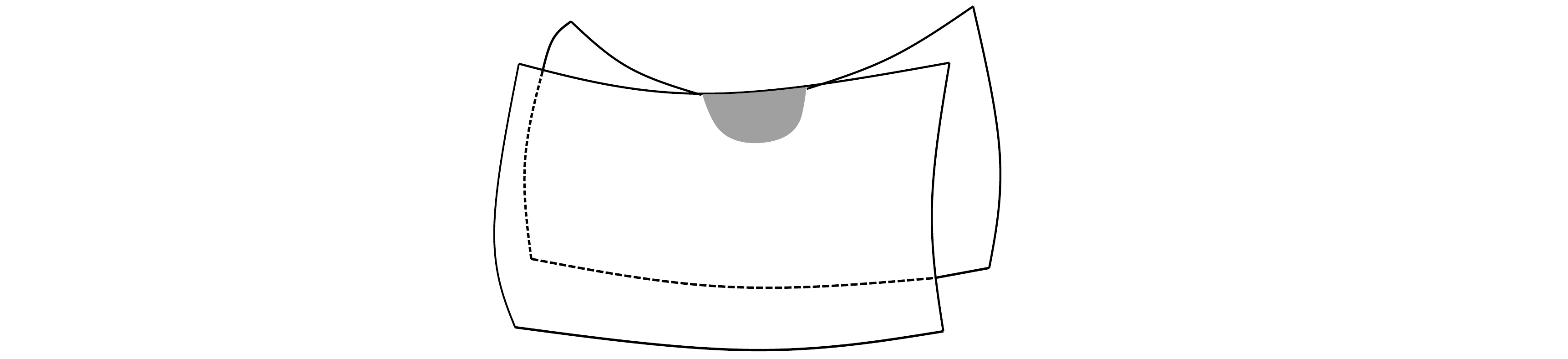
	\caption{The basic cobordism} 
	\label{3.2}
\end{figure} 

Since we work in $\mathbb{F}$ each square in the cube commute, and we can define the total complex $F(D)$. More precisely. let $F(D)_i= \bigoplus_{|J|=i} J(D_J)$, and the differential $d^i: F(D)_i \longrightarrow F(D)_{i+1}$ is induced by the edge homomorphisms defined above.  Fig \ref{3.4}  explains our construction of $F(D)$.

\begin{figure}[h]
	\def\svgwidth{\columnwidth}
\begingroup%
  \makeatletter%
  \providecommand\color[2][]{%
    \errmessage{(Inkscape) Color is used for the text in Inkscape, but the package 'color.sty' is not loaded}%
    \renewcommand\color[2][]{}%
  }%
  \providecommand\transparent[1]{%
    \errmessage{(Inkscape) Transparency is used (non-zero) for the text in Inkscape, but the package 'transparent.sty' is not loaded}%
    \renewcommand\transparent[1]{}%
  }%
  \providecommand\rotatebox[2]{#2}%
  \newcommand*\fsize{\dimexpr\f@size pt\relax}%
  \newcommand*\lineheight[1]{\fontsize{\fsize}{#1\fsize}\selectfont}%
  \ifx\svgwidth\undefined%
    \setlength{\unitlength}{841.88976378bp}%
    \ifx\svgscale\undefined%
      \relax%
    \else%
      \setlength{\unitlength}{\unitlength * \real{\svgscale}}%
    \fi%
  \else%
    \setlength{\unitlength}{\svgwidth}%
  \fi%
  \global\let\svgwidth\undefined%
  \global\let\svgscale\undefined%
  \makeatother%
  \begin{picture}(1,0.2020202)%
    \lineheight{1}%
    \setlength\tabcolsep{0pt}%
    \put(0,0){\includegraphics[width=\unitlength,page=1]{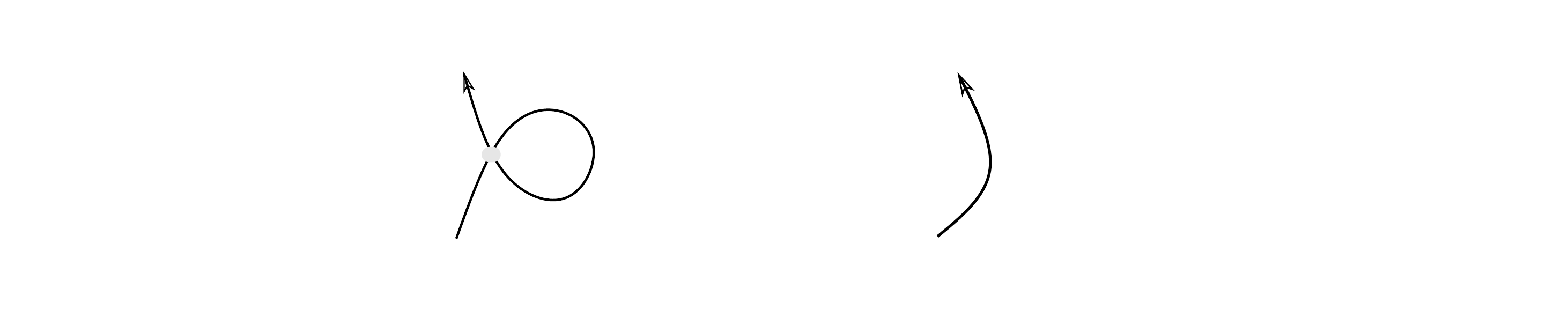}}%
    \put(0.18683431,0.09187924){\color[rgb]{0,0,0}\makebox(0,0)[lt]{\lineheight{1.25}\smash{\begin{tabular}[t]{l}D =\end{tabular}}}}%
    \put(0.50915276,0.09435383){\color[rgb]{0,0,0}\makebox(0,0)[lt]{\lineheight{1.25}\smash{\begin{tabular}[t]{l}D'=\end{tabular}}}}%
    \put(0,0){\includegraphics[width=\unitlength,page=2]{3.4.pdf}}%
  \end{picture}%
\endgroup%
 
	\caption{virtual Reidemeister move I}
	\label{3.4}
\end{figure}

\begin{figure}[h]
	\def\svgwidth{\columnwidth}
\begingroup%
  \makeatletter%
  \providecommand\color[2][]{%
    \errmessage{(Inkscape) Color is used for the text in Inkscape, but the package 'color.sty' is not loaded}%
    \renewcommand\color[2][]{}%
  }%
  \providecommand\transparent[1]{%
    \errmessage{(Inkscape) Transparency is used (non-zero) for the text in Inkscape, but the package 'transparent.sty' is not loaded}%
    \renewcommand\transparent[1]{}%
  }%
  \providecommand\rotatebox[2]{#2}%
  \newcommand*\fsize{\dimexpr\f@size pt\relax}%
  \newcommand*\lineheight[1]{\fontsize{\fsize}{#1\fsize}\selectfont}%
  \ifx\svgwidth\undefined%
    \setlength{\unitlength}{935.43307087bp}%
    \ifx\svgscale\undefined%
      \relax%
    \else%
      \setlength{\unitlength}{\unitlength * \real{\svgscale}}%
    \fi%
  \else%
    \setlength{\unitlength}{\svgwidth}%
  \fi%
  \global\let\svgwidth\undefined%
  \global\let\svgscale\undefined%
  \makeatother%
  \begin{picture}(1,0.12121212)%
    \lineheight{1}%
    \setlength\tabcolsep{0pt}%
    \put(0,0){\includegraphics[width=\unitlength,page=1]{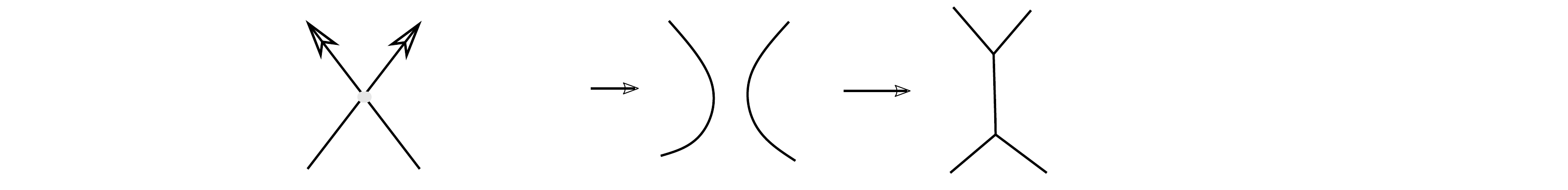}}%
    \put(0.28744488,0.0594687){\color[rgb]{0,0,0}\makebox(0,0)[lt]{\lineheight{1.25}\smash{\begin{tabular}[t]{l}=\\\end{tabular}}}}%
    \put(0.35131549,0.05788248){\color[rgb]{0,0,0}\makebox(0,0)[lt]{\lineheight{1.25}\smash{\begin{tabular}[t]{l}0\end{tabular}}}}%
    \put(0.75137757,0.05629626){\color[rgb]{0,0,0}\makebox(0,0)[lt]{\lineheight{1.25}\smash{\begin{tabular}[t]{l}0\end{tabular}}}}%
    \put(0.31821892,0.0594687){\color[rgb]{0,0,0}\makebox(0,0)[lt]{\lineheight{1.25}\smash{\begin{tabular}[t]{l}[\end{tabular}}}}%
    \put(0.78447416,0.05735374){\color[rgb]{0,0,0}\makebox(0,0)[lt]{\lineheight{1.25}\smash{\begin{tabular}[t]{l}]\end{tabular}}}}%
    \put(0,0){\includegraphics[width=\unitlength,page=2]{3.3.pdf}}%
  \end{picture}%
\endgroup%
 
	\label{3.3}
\end{figure}

\begin{theorem}
	The homotopy type of $F(D)$ is an invariant of $D$, i.e. it is invariant under the virtual Reidemeister moves.
\end{theorem}
\begin{proof}
	The argument is similar to the proof of invariance of Khovanov' $sl_3$ homology under Reidemeister moves, since all the local relations used there have their counterparts in our theory(Section \ref{sec3}). For example, to prove the invariance under virtual Reidemeister moves I(Fig.\ref{3.4}), we borrow the diagram from \cite{MR2336253}, Figure 7, and define the chain maps $f, g$, the chain homotopy $h$ as illustrated in the graph. The statement there carries verbatim to show that $f, g$ are chain homotopy equivalences to each other.

	\begin{figure}[h]
		\def\svgwidth{\columnwidth}
		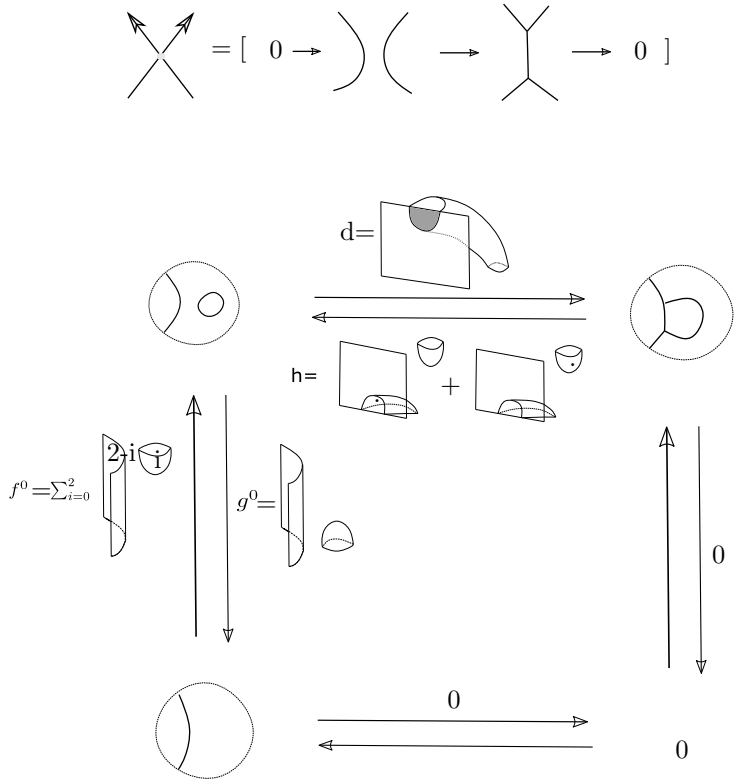 
		\label{graph3.5}
		\caption{Chain homotopy equivalance under Virtual Reidemeister move I}
	\end{figure}
	
\end{proof}

 An immediate corollary is

\begin{corollary} \label{c5}
	The Euler characteristic $e(D)$ of $F(D)$ is an invariant of $D$.
\end{corollary}

\textit{Proof of Theorem \ref{thm1}}    

     Let $D$ be an oriented virtual web. We proved that the Euler characteristic of $F(D)$ is equal to the Penrose number of $D$. In particular, when $D$ has no crossings, the dimension of $J(D)$ is equal to the number of Tait colorings of $D$ by Prop   \ref{prop3}.          
     
     If $D$ has no vertices, then it is a union of k circles(Lemma \ref{l1}).  By Corollary  \ref{c1}  \ref{c5}, $e(D)=e(U^k)=dim J(U^k)=3^k=P(D)$ , where $U^k$ is a disjoint union of k circles.  Furthermore, from the definition of $F(D)$, we have that $e(D)$ satisfies the local relation in Prop. \ref{prop2}. This identifies $e(D)$ with $P(D)$. In particular, when $D$ is a web, $dim J(D)= e(D)= P(D) =T(D)$, by Prop \ref{prop3}.

\bibliography{zz}
\nocite{*}

\begin{figure}[h]
	\def\svgwidth{\columnwidth}
	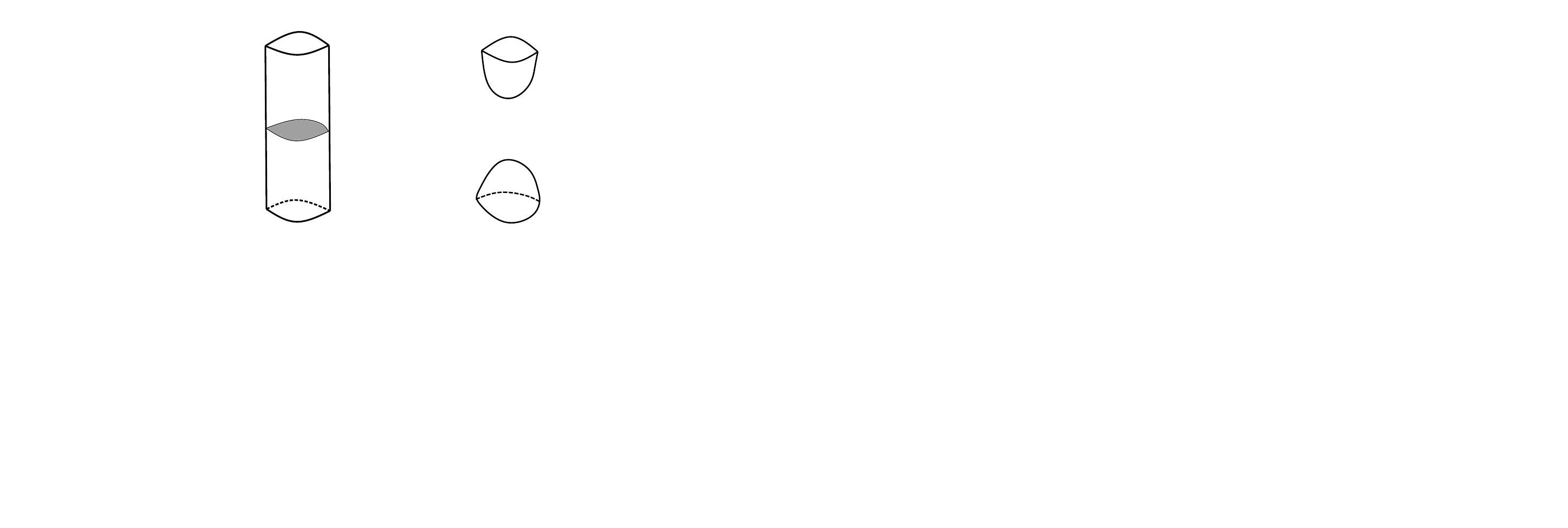
	\label{6.4}
\end{figure}

\begin{figure}[h]
	\def\svgwidth{\columnwidth}
	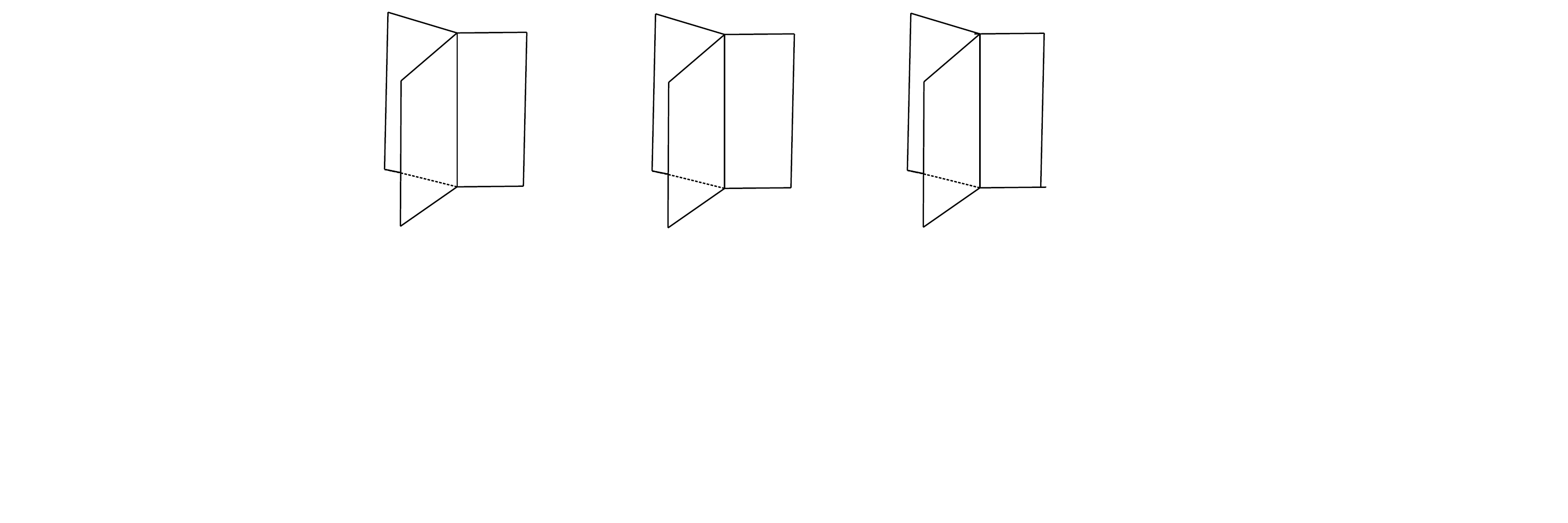
	\label{6.4.2}
\end{figure}

\end{document}